\documentclass[oneside, 10pt, a4paper]{article} 

\usepackage[utf8]{inputenc} 

\usepackage{geometry} 
\usepackage{graphicx} 

\usepackage[all]{xy}
\usepackage{amsmath}
\usepackage{amsthm}
\usepackage[shortlabels, inline]{enumitem}
\usepackage{mathrsfs}
\usepackage{amsfonts}
\usepackage{amssymb}
\usepackage{tikz}
\usepackage{tikz-cd} 
\usepackage{url}
\usepackage{mathtools}
\usepackage{rotating}
\usepackage[thinc]{esdiff}
\usepackage[english]{isodate}
\usepackage{csquotes}
\usepackage[most]{tcolorbox}
\tcbset{colback=yellow!10!white, colframe=red!50!black, 
        highlight math style= {enhanced, 
            colframe=red,colback=red!10!white,boxsep=0pt}
        }
\usetikzlibrary{matrix}
\usetikzlibrary{babel}
\usepackage{cancel}
\usepackage{booktabs} 
\usepackage{array} 
\usepackage{verbatim} 
\usepackage{subfig} 
\usepackage{hyperref} 
\hypersetup{
    colorlinks,
    citecolor=black,
    filecolor=black,
    linkcolor=black,
    urlcolor=black
}

\usepackage{fancyhdr} 
\usepackage{sectsty}
\allsectionsfont{\sffamily\mdseries\upshape} 

\usepackage[nottoc,notlof,notlot]{tocbibind} 
\usepackage[titles,subfigure]{tocloft} 

\makeatletter
\DeclareRobustCommand{\tvdots}{%
  \vbox{\baselineskip4\p@\lineskiplimit\z@\kern0\p@\hbox{.}\hbox{.}\hbox{.}}}
\makeatother
\makeatletter\@addtoreset{chapter}{part}\makeatother%
\tikzset{
  symbol/.style={
    draw=none,
    every to/.append style={
      edge node={node [sloped, allow upside down, auto=false]{$#1$}}}
  }
}
\makeatletter
\DeclareRobustCommand{\tvdots}{%
  \vbox{\baselineskip4\p@\lineskiplimit\z@\kern0\p@\hbox{.}\hbox{.}\hbox{.}}}
\makeatother

\theoremstyle{plain}
\newtheorem{thm}{Theorem}[section]
\newtheorem{prop}[thm]{Proposition}

\newtheorem{lem}[thm]{Lemma}
\theoremstyle{definition}
\newtheorem{defn}[thm]{Definition}

\newtheorem{ex}[thm]{Example}
\theoremstyle{remark}

\newtheorem*{rem}{Remark}

\newcommand{\map}[3]{#1\colon #2 \to #3}

\newcommand{\mb}{\mathbb}
\newcommand{\gam}{\gamma}
\newcommand{\del}{\delta}

\newcommand{\pr}{\prime}

\newcommand{\nCr}[2]{\left(\begin{matrix}#1\\#2\end{matrix}\right)}

\newcommand{\alf}{\alpha}
\newcommand{\mc}{\mathcal}

\newcommand{\im}[1]{\mbox{im}(#1)}

\title{The spectral sequence of a polycomplex}
\author{Andrew Phimister}
\date{} 

\begin{document}
\maketitle

\begin{abstract}
We give a very brief introduction to the machinery of spectral sequences, including the spectral sequence of a bicomplex. We then briefly introduce a generalisation of the spectral sequences of a bicomplex to the spectral sequences of abritrary higher complexes. 
\end{abstract}

\paragraph{\textbf{Note}} This preprint is a modified version of Chapter 5 of the author's PhD thesis~\cite{phimphd2025}, to which the reader is directed should they require additional information and application of this material.

\section{Introduction} 

Spectral sequences, in short, provide a powerful tool for calculating homology and cohomology in various circumstances, as well as providing insight into more purely algebraic topics. A spectral sequence can be thought of a sequence of approximations of (co)homology that, under favourable conditions, eventually give a precise calculation. They have been used in algebraic and topological contexts to great effect, such as with the Leray-Serre spectral sequence or the Eilenberg-Moore spectral sequence, information about which can be found in McCleary's textbook on the topic~\cite{McClearyJohn2001Augt}.

A rich source of nicely behaving spectral sequences is the spectral sequences associated to the filtrations on the total complex of a bicomplex. However, a natural generalisation of bicomplexes, which can be thought of as a 2 dimensional cochain complex, is a complex of arbitrary higher dimensions. These higher complexes, or \emph{polycomplexes}, also have a total complex and we can again look at filtrations on this total complex to obtain a spectral sequence. This has remained largely unexplored in the literature, and so we begin to rectify that here.   

We will work exclusively with $\Bbbk$-vector spaces. 

\section{Spectral Sequences}

Let $(V^*, d)$ be a cochain complex of vector spaces 
\[
\begin{tikzcd}
\cdots \ar{r} & V^{n-1} \ar{r}{d} & V^n \ar{r}{d} & V^{n+1} \ar{r}{d} & \cdots
\end{tikzcd}
\]
We will assume that $V^*$ is locally finite, that is, each $V^n$ is a finite dimensional vector space. We want to calculate the cohomology $H^*$ of this cochain complex, where
\[
	H^n(V^*) = \frac{\ker{(\map{d}{V^n}{V^{n+1}})}}{\im{\map{d}{V^{n-1}}{V^n}}}
\]
however in general this complex might have extra structure that makes calculations more complicated. If we are fortunate enough to have a ``graded'' complex where each $V^n = \bigoplus_\alf V^n_\alf$ and $d(V^n_\alf)\subset V^{n+1}_\alf$ then conveniently this induces a grading on the cohomology, and so we have
\[
	H^n(V^*) = H^n(\bigoplus_\alf V^*_\alf) \simeq \bigoplus_\alf H^n(V^*_\alf)
\]
However it's often the case that instead of a grading, we have a filtration $F$ on our cochain complex. So let $F$ be a decreasing filtraton on $V^*$
\[
\begin{tikzcd}
\{0\} \subset \cdots \subset F^{n}V^* \subset F^{n-1}V^* \subset \cdots \subset F^{0}V^* =  V^*
\end{tikzcd}
\]
such that $d(F^pV^n) \subset F^pV^{n+1}$. A cochain complex with such a filtration may be referred to as a \emph{filtered cochain complex} and be denoted $(V^*, F)$. We can visualise this scenario in the following way
\[
\begin{tikzcd}
\cdots \ar{r} &V^{n-1} \ar{r}{d} & V^n \ar{r}{d}& V^{n+1} \ar{r} & \cdots \\   
\cdots \ar{r} &\arrow[u,symbol=\subset] F^1V^{n-1} \arrow{r}{d}&\arrow[u,symbol=\subset] F^1V^n \arrow{r}{d}&\arrow[u,symbol=\subset]F^1V^{n+1} \arrow{r}& \cdots \\
{}&\tvdots\arrow[u,symbol=\subset] &\tvdots \arrow[u,symbol=\subset]&\tvdots\arrow[u,symbol=\subset]& {}\\
\cdots \ar{r} &\arrow[u,symbol=\subset] F^pV^{n-1} \arrow{r}{d}&\arrow[u,symbol=\subset] F^pV^n \arrow{r}{d}&\arrow[u,symbol=\subset]F^pV^{n+1} \arrow{r}& \cdots \\
{}&\tvdots\arrow[u,symbol=\subset] &\tvdots \arrow[u,symbol=\subset]&\tvdots\arrow[u,symbol=\subset]& {}
\end{tikzcd}
\]
where each row is a cochain complex. Just as with a graded cochain complex, the filtration on $V^*$ induces a filtration on the cohomology $H^*$
\[
\begin{tikzcd}
\cdots \subset F^{p+1}H^* \subset F^{p}H^* \subset F^{p-1}H^* \subset \cdots \subset F^0H^* = H^*
\end{tikzcd}
\]
and therefore we can arrange the cohomology in a similar grid 
\[
\begin{tikzcd}
\cdots  &H^{n-1}  & H^n & H^{n+1}  & \cdots \\   
\cdots  &\arrow[u,symbol=\subset] F^1H^{n-1} &\arrow[u,symbol=\subset] F^1H^n &\arrow[u,symbol=\subset]F^1H^{n+1} & \cdots \\
{}&\tvdots\arrow[u,symbol=\subset] &\tvdots \arrow[u,symbol=\subset]&\tvdots\arrow[u,symbol=\subset]& {}\\
\cdots &\arrow[u,symbol=\subset] F^pH^{n-1} &\arrow[u,symbol=\subset] F^pH^n &\arrow[u,symbol=\subset]F^pH^{n+1}& \cdots \\
{}&\tvdots\arrow[u,symbol=\subset] &\tvdots \arrow[u,symbol=\subset]&\tvdots\arrow[u,symbol=\subset]& {}
\end{tikzcd}
\]
Now, we can't expect the cohomology to split as a direct sum of the subspaces in the filtration, as it did with the grading. Any direct sum $\bigoplus F^pH^n$ would cause a severe over-representation of elements. 

We want to find some way of splitting the cohomology into a direct sum, so that we can simply calculate each part individually, and then collect our terms. We can do this in the following way, by defining 
\[
	E^{p, q}_0(H^*) = \frac{F^pH^{p+q}}{F^{p+1}H^{p+q}}
\]
where $F^{n+1}H^n = \{0\}$, then we see that 
\[
	H^n(V) \simeq \bigoplus_{p+q=n}E^{p, q}_0(H^*)
\]
Now the question is: If we apply these same ideas to our cochain complex $(V^*, d)$ and then calculate the cohomology, do we obtain the terms $E^{p, q}_0(H^*)$?

Unfortunately, we do not. However it does give us, in a very real sense, an \emph{approximation} of the cohomology of $(V^*, d)$, and is the first step in obtaining a precise calculation of $H^*$.

\begin{defn}
Consider our cochain complex $(V^*, d)$, and the filtration $F$ on $V^*$. Then the \emph{associated graded vector space} $E^{p, q}_0$ of $V^*$ is
\[
	E^{p, q}_0 = \frac{F^pV^{p+q}}{F^{p+1}V^{p+q}}
\]
The \emph{$E_0$-page} is the bigraded object $E^{*,*}_0$ made up of $E^{p,q}_0$ at each point $(p,q)$.
\end{defn}

Then, as before, we can see that
\[
	V^n \simeq \bigoplus_{p+q=n}E^{p,q}_0
\]
\begin{rem}
Note that this identification relies on $V^*$ being locally finite. For infinite dimensional vector spaces it is not true in general that for a subspace $U\subset V^*$ we have $(V^*/U)\oplus U \simeq V^*$. 

For example, take $\Bbbk[x]$ and consider the subspace $U=\mbox{Span}(x^n)$. Then $\Bbbk[x]/U$ is an $n$-dimensional vector space and $U$ is a $1$-dimensional subspace of $\Bbbk[x]$. Therefore $(\Bbbk[x]/U)\oplus U$ is an $(n+1)$-dimensional vector space and so cannot be isomorphic to $\Bbbk[x]$.

However it should be noted that spectral sequences can be defined in any abelian category, see~\cite{McClearyJohn2001Augt} or~\cite{weibel} for information on spectral sequences in abelian categories.
\end{rem}

Now, recall that the filtration induces a differential $\map{d}{F^pV^n}{F^pV^{n+1}}$ for each $p$. Therefore we have an induced differential 
\[
	\map{d_0}{E^{p,q}_0}{E^{p,q+1}_0}
\]
and so we are able to take the cohomology $H(E^{p,q}_0)$. We set
\[
	E^{p, q}_1 = H(E^{p,q}_0)
\]
and the bigraded object $E^{*, *}_1$ is called the \emph{$E_1$-page}. Then we have another induced differential
\[
	\map{d_1}{E^{p,q}_1}{E^{p+1,q}_1}.
\]
The details for how precisely this map is induced can be found in \cite[Proof of Theorem 2.6]{McClearyJohn2001Augt}. Then again we can take the cohomology and set 
\[
	E^{p, q}_2 =  H(E^{p,q}_1)
\]
Continuing in this way we obtain a sequence of bigraded objects $E^{*,*}_r$ for $r\ge 0$. Which leads us to our definition of a spectral sequence. 
\begin{defn}
A (first quadrant, cohomological) \emph{spectral sequence} is a sequence of bigraded vectors spaces $E^{*, *}_r$ for $r\ge 0$. So that for each $r$ and $p,q\ge 0$ we have a vector space $E^{p, q}_r$. Each $E^{*, *}_r$ is equipped with a differential $d_r$ of bidegree $(r, 1-r)$
\[
	\map{d_r}{E^{p, q}_r}{E^{p+r, q-r+1}_r}
\]
and for all $r\ge 1$ we have $E^{*, *}_{r+1}\simeq H(E^{*, *}_r)$, that is
\[
	E^{p, q}_{r+1} = \frac{\map{\ker{(d_r)}}{E^{p, q}_r}{E^{p+r, q-r+1}_r}}{\map{\im{d_r}}{E^{p-r, q+r-1}_r}{E^{p, q}_r}}
\]
We call the $r$-th stage of such an object its \emph{$E_r$-page}.
\end{defn}

The idea is that each page $E_r$ gives us a closer and closer approximation to the cohomology of our original cochain complex $(V^*, d)$, and so we hope that the limit of these pages will precisely calculate $H^*(V^*)$. We will see that in certain nice cases this can be guaranteed.

For first quadrant spectral sequences, the $E_r$-page can be visualised as a lattice in the first quadrant with each lattice point a vector space and the differentials as arrows, for example see below for $r = 2$
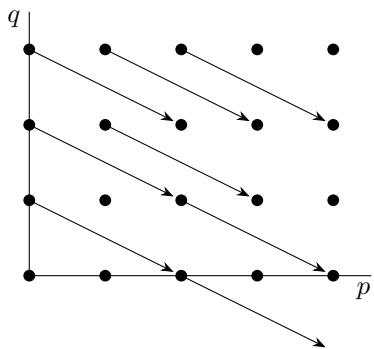
\begin{center}
 \begin{tikzpicture}[>=Stealth]
 \draw (0,0) -- (4.5,0);
\draw (0,0) -- (0,3.5);
\draw (4.4, -0.2) node {$p$}
(-0.2,3.4) node {$q$};
 \filldraw (0,0) circle [radius=2pt]
 (1,0) circle [radius=2pt]
 (2,0) circle [radius=2pt]
 (3,0) circle [radius=2pt]
 (4,0) circle [radius=2pt]
 (0,1) circle [radius=2pt]
 (1,1) circle [radius=2pt]
 (2,1) circle [radius=2pt]
 (3,1) circle [radius=2pt]
 (4,1) circle [radius=2pt]
 (0,2) circle [radius=2pt]
 (1,2) circle [radius=2pt]
 (2,2) circle [radius=2pt]
 (3,2) circle [radius=2pt]
 (4,2) circle [radius=2pt]
 (0,3) circle [radius=2pt]
 (1,3) circle [radius=2pt]
 (2,3) circle [radius=2pt]
 (3,3) circle [radius=2pt]
 (4,3) circle [radius=2pt];
\draw[->] (0, 3) -- (1.9, 2.05);
\draw[->] (1, 3) -- (2.9, 2.05);
\draw[->] (2, 3) -- (3.9, 2.05);
\draw[->] (1, 2) -- (2.9, 1.05);
\draw[->] (0, 2) -- (1.9, 1.05);
\draw[->] (2, 1) -- (3.9, 0.05);
\draw[->] (0, 1) -- (1.9, 0.05);
\draw[->] (2, 0) -- (3.9, -0.95);
 \end{tikzpicture}
\end{center}

Note that for $E^{p, q}_r$, when $r>\mbox{max}(p, q+1)$ then the differential $d_r$ is trivial, since $q+1-r<0$ and so $E^{p+r, q+1-r}_r = \{0\}$ and therefore $\ker{(d_r)} = E^{p, q}_r$. Similarly, since $p-r<0$ then $E^{p-r, q+r-1}_r = \{0\}$ and so $\im{d_r}=\{0\}$. Therefore
\[
	E^{p, q}_{r+1} = E^{p, q}_r
\]
and continuing in this way we have $E^{p, q}_{r+k} = E^{p, q}_r$ for $k\ge 0$. We denote this vector space by $E^{p, q}_\infty$.

\begin{defn}
If we have a spectral sequence $E^{*, *}_r$ such that $E^{p, q}_\infty \simeq E^{p, q}_0(H^*)$, then we say that it \emph{converges} to $H^*(V^*)$ and write $E^{p, q}_r \Rightarrow H^{p+q}(V^*)$.
\end{defn}

It is typical to put specific emphasis on the form of the $E_2$ page for a given spectral sequence. This is because the $E_2$ will often be the first interesting page of a spectral sequence. With this in mind, if we have information on the form that the components of the second page will take, then we will usually represent a convergent spectral sequence by
\[
	E_2^{p, q}\Rightarrow H^{p+q}(V^*)
\]
as a short hand for both describing $E_2$ and signifying that the spectral sequence converges.

We will now discuss the aforementioned nice cases that guarantee the convergence of our spectral sequence. In aid of this we will introduce some basic terminology about filtrations and spectral sequences. 

\begin{defn}
A decreasing filtration $F$ is \emph{bounded} if for each $n$ there are integers $s> t$ such that $F^sV^n = \{0\}$ and $F^tV^n = V^n$. If for each $n$ we have $F^{n+1}V^n = \{0\}$ and $F^0V^n = V^n$ then we say that the filtration is \emph{canonically bounded}.
\end{defn}

Clearly, canonically bounded is just a special case of bounded, where $s = n+1$ and $t= 0$ for each $n\ge 0$. 

\begin{defn}
Let $E^{*,*}_r$ be the $r$-th page of a spectral sequence of cohomological type, and for each $E^{p, q}_r$ call $p+q$ the \emph{total degree} of $E^{p, q}_r$. This spectral sequence is \emph{bounded} if there are only finitely many non-zero terms in each total degree of $E^{*,*}_r$. 
\end{defn}

Now we will state a criteria for convergence (there are others but for our purposes just this one will do).

\begin{thm}[Convergence Theorem]
\label{conv}
Suppose that the filtration $F$ on a cochain complex $(V^*, d)$ is bounded. Then the spectral sequence arising from $F$ is bounded and converges to $H^*(V^*)$.
\end{thm}

The proof for this theorem can be found here~\cite[page 34]{McClearyJohn2001Augt}. And as we've noted, canonically bounded filtrations are examples of bounded filtrations, and therefore the spectral sequence arising from them converges to the cohomology of the cochain complex. This is important for the example we are about to explore in the next section.

\begin{prop}
\label{edge hom}
Suppose we have a canonically bounded filtration $F$ of $V^*$ that induces a first quadrant cohomological spectral sequence $E$. Then there are maps 
\begin{align*}
	& H^n(V^*) \to E^{0,n}_2\\
	& E^{n,0}_2\to H^n(V^*)
\end{align*}
for each $n\ge 0$. These maps are called the \emph{edge morphisms}.
\end{prop}

\begin{proof}
In general we have 
\[
	E^{p, q}_{r+1} = \frac{\map{\ker{(d_r)}}{E^{p, q}_r}{E^{p+r, q-r+1}_r}}{\map{\im{d_r}}{E^{p-r, q+r-1}_r}{E^{p, q}_r}}
\]
and for first quadrant spectral sequences we have $E_{r}^{p, q}=0$ for $p<0$ or $q<0$. So for each $n, r\ge 0$ we have
\[
	E^{0, n}_{r+1} = \frac{\map{\ker{(d_r)}}{E^{0, n}_r}{E^{r, n-r+1}_r}}{\map{\im{d_r}}{E^{-r, n+r-1}_r}{E^{0, n}_r}} \simeq \map{\ker{(d_r)}}{E^{0, n}_r}{E^{r, n-r+1}_r}\subseteq E_r^{0, n}.
\]
Since for first quadrant spectral sequences each $E^{p, q}_r$ stabilizes for sufficiently large $r$, that is $E^{p, q}_r\simeq E^{p, q}_{r+1}$, then we obtain a tower of inclusions
\[
	E^{0,n}_\infty\subseteq\cdots\subseteq E^{0,n}_2.
\]
Since this spectral sequence has arisen from a canonically bounded filtration $F$ then $H^n(V^*)= F^0H^n(V^*)$. However
\[
	E_\infty^{0, n} = \frac{F^0H^n(V^*)}{F^1H^n(V^*)}
\]
and therefore there is a surjection 
\[
	H^n(V^*) = F^0H^n(V^*) \to \frac{F^0H^n(V^*)}{F^1H^n(V^*)} = E_\infty^{0, n}.
\]
Therefore composing this surjection with the tower of inclusions above we get a map
\[
	H^n(V^*)\to E_2^{0, n}.
\]
Now since for $r\ge 2$ we have $1-r<0$, then $E_r^{n+r,1-r} = 0$ and therefore
\[
	E^{n, 0}_{r+1} = \frac{\map{\ker{(d_r)}}{E^{n, 0}_r}{E^{n+r, 1-r}_r}}{\map{\im{d_r}}{E^{n-r, r-1}_r}{E^{n, 0}_r}} \simeq E^{n,0}_r/{\map{\im{d_r}}{E^{n-r, r-1}_r}{E^{n, 0}_r}}.
\] 
Hence, there is a surjection $E_r^{n,0}\to E^{n, 0}_{r+1}$ and therefore we have a surjection $E^{n,0}_2\to E^{n,0}_\infty$. Now, since $F$ is canonically bounded then $F^{n+1}H^n(V^*) = 0$ and so 
\[
	E_\infty^{n, 0} = \frac{F^nH^n(V^*)}{F^{n+1}H^n(V^*)}\simeq F^nH^n(V^*) \subset H^n(V^*)
\]
and so composing the surjection $E^{n,0}_2\to E^{n,0}_\infty$ with the inclusion $E_\infty^{n, 0}\to H^n(V^*)$ we obtain a map 
\[
	E^{n,0}_2\to H^n(V^*)
\]
as required.
\end{proof}

\begin{defn}
Let $E$ and $E^\prime$ be two spectral sequences. Then a \emph{morphism of spectral sequences} $\map{f}{E}{E^\prime}$ is a sequence of bidgree $(0,0)$ homomorphisms $\map{f_r^{p,q}}{E_r^{p,q}}{{E^\prime}_r^{p,q}}$ such that $f$ commutes with the differentials of $E$ and $E^\prime$ and each $f_{r+1}$ is induced by $f_r$ on cohomology.
\end{defn}

\begin{prop}
\label{spec morph prop}
Consider two filtered cochain complexes $(V^*, F)$ and $(W^*, F^\pr)$. Let $\map{\phi}{V^*}{W^*}$ be a cochain map satisfying $\phi(F^pV^n)\subset {F^\pr}^pW^n$ for all $n$. Then $\phi$ induces a morphism of spectral sequences. 
\end{prop}

\begin{proof}
See~\cite[Theorem 3.5]{McClearyJohn2001Augt}.
\end{proof}

\section{The Spectral Sequence of a Bicomplex}

\begin{defn}
A \emph{bicomplex} is a bigraded vector space $V^{**}$ with two maps $\map{d_v}{V^{p, q}}{V^{p, q+1}}$ and $\map{d_h}{V^{p, q}}{V^{p+1, q}}$ such that
\[
	d_h^2 = d_v^2 = d_vd_h + d_hd_v = 0
\]
so that $d_h$ and $d_v$ are differentials, and they anti-commute. 
\end{defn}
So we have the following picture
\[
\begin{tikzcd}
                          &\tvdots                                                                                        & \tvdots                                                                     & \tvdots                                                           & \\   
\cdots \arrow{r}& \arrow{r}{d_h} V^{p-1, q+1} \arrow{u}                                         & \arrow{r}{d_h} V^{p, q+1} \arrow{u}       & \arrow{r} V^{p+1, q+1}  \arrow{u}           &  \cdots \\
\cdots \arrow{r}& \arrow{r}{d_h} V^{p-1, q} \arrow{u}{d_v}     &\arrow{r}[swap]{d_h} V^{p, q} \arrow{u}[swap]{d_v}   &  \arrow{r} V^{p+1, q} \arrow{u}[swap]{d_v}       & \cdots  \\
\cdots \arrow{r}& \arrow{r}[swap]{d_h} V^{p-1, q-1}\arrow{u}{d_v}            & \arrow{r}[swap]{d_h} V^{p, q-1} \arrow{u}{d_v}          &\arrow{r} V^{p+1, q-1} \arrow{u}[swap]{d_v}     & \cdots \\
                          &\tvdots \arrow{u}                                                                        &\tvdots \arrow{u}                                                     & \tvdots  \arrow{u}                                          & 
\end{tikzcd}
\]

We will only consider first quadrant bicomplexes, so that $p, q\ge 0$.

\begin{defn}
Let $V^{**}$ be a first quadrant bicomplex. Then the \emph{total complex} $\mc{T}(V^{**})$ of $V^{**}$ is the cochain complex defined by
\[
	\mc{T}(V^{**})^n = \bigoplus_{p+q=n}V^{p,q}
\]
and differential $d = d_h+d_v$.
\end{defn}

So then $\mc{T}(V^{**})$ is the cochain complex
\[
\begin{tikzcd}
V^{0,0} \ar{r}{d} & V^{1,0}\oplus V^{0,1} \ar{r}{d} & V^{2,0}\oplus V^{1,1}\oplus V^{0,2} \ar{r}{d} &  \cdots
\end{tikzcd}
\]
We can define two filtrations on $\mc{T}(V^{**})$, one coming from the columns of $V^{**}$ and the other from the rows. 	

Let $'F$ be the filtration on $\mc{T}(V^{**})$ we obtain by deleting the first $s$ columns of $V^{**}$ and then finding the total complex, that is $'F^s\mc{T}(V^{**})$ is the total complex of
\[
\begin{tikzcd}
               & \tvdots                           &   \tvdots                      & \tvdots                           &     \\
   0  \ar{r}   &  V^{s,2} \ar{u}\ar{r}    &  V^{s+1, 2}\ar{u} \ar{r}    & V^{s+2, 2}\ar{u}\ar{r}     & \cdots  \\
    0  \ar{r}   &  V^{s,1} \ar{u}\ar{r}    &  V^{s+1, 1}\ar{u}\ar{r}    & V^{s+2,1} \ar{u}\ar{r}     & \cdots  \\
    0  \ar{r}   &  V^{s,0} \ar{u}\ar{r}    &  V^{s+1, 0}\ar{u} \ar{r}    & V^{s+2,0} \ar{u}\ar{r}     & \cdots
\end{tikzcd}
\]
and therefore 
\[
	'F^s\mc{T}(V^{**})^n = \bigoplus_{p+q = n, \, p\ge s} V^{p, q}
\]
We will call $'F$ the \emph{vertical filtration}. We can see that $'F^{n+1}\mc{T}(V^{**})^n = 0$ and $'F^0\mc{T}(V^{**})^n = \mc{T}(V^{**})^n$. Therefore $'F$ is a canonically bounded filtration of $\mc{T}(V^{**})$, and so by Theorem~\ref{conv} the spectral sequence $'E^{*,*}$ arising from this filtration is bounded and converges to $H^*(\mc{T}(V^{**}))$.

Similarly, the filtration $''F$ arising from the rows, which we will call the \emph{horizontal filtration}, given by
\[
	 ''F^t\mc{T}(V^{**})^n = \bigoplus_{p+q = n, \, q\ge t} V^{p, q}
\]
is canonically bounded, and so the spectral sequence $''E^{*,*}$ also converges to $H^*(\mc{T}(V^{**}))$.

Note that 
\[
	'E^{p, q}_0 \simeq V^{p ,q}, \quad ''E^{p, q}_0 \simeq V^{q ,p}
\]
and therefore we can identify $'E_1^{p, q} \simeq H^q(V^{p,*})$ and $''E_1^{p, q} \simeq H^q(V^{*,p})$. That is, $'E_1^{p, q}$ is the vertical cohomology of $V^{**}$ and $''E_1^{p, q}$ is the horizontal cohomology of $V^{**}$.
\begin{prop}
Let $'E^{*,*}$ and $''E^{*,*}$ be the spectral sequences arising from the filtrations $'F$ and $''F$ respectively. Then 
\[
	'E^{p, q}_2 \simeq H_h^p(H_v^q(V^{**}))
\]
and
\[
	''E^{p,q}_2 \simeq H_v^p(H_h^q(V^{**}))
\]
where $H_v$ and $H_h$ is the vertical and horizontal cohomology of the bicomplex $V^{**}$ respectively.  
\end{prop}

\begin{ex}[Proof of the Snake Lemma]
This proof of the snake lemma using spectral sequences is due to Vakil~\cite{snakelemm}. Consider the following diagram 
\[
\begin{tikzcd}
0 \ar{r} & D \ar{r}                     & E \ar{r}                        & F \ar{r}                       & 0 \\
0 \ar{r} & A \ar{u}{\alf} \ar{r} & B \ar{u}{\beta} \ar{r} & C \ar{u}{\gam} \ar{r} & 0
\end{tikzcd}
\]
with exact rows and commuting squares. We want to show that there is an exact sequence 
\[
\begin{tikzcd}[cramped, sep=small]
0 \ar{r} & \ker{(\alf)} \ar{r} & \ker{(\beta)} \ar{r} & \ker{(\gam)} \ar{r} & \mbox{coker}(\alf) \ar{r} & \mbox{coker}(\beta) \ar{r} & \mbox{coker}(\gam) \ar{r} & 0
\end{tikzcd}
\]
Consider the horizontal filtration $''F$. The $''E_0$-page is given by   
\[
\begin{tikzcd}[cramped, sep=small]
0  & 0       \\
F \ar{u} & C \ar{u}\\
E \ar{u} & B \ar{u} \\
D \ar{u} & A \ar{u} \\
0 \ar{u}& 0 \ar{u}
\end{tikzcd}
\]
Since the rows of the original diagram are exact then $''E_1^{p, q}= 0$ for all $(p, q)$, and therefore $''E_\infty^{p, q} = 0$.

Now consider the vertical filtration $'F$. Since $''E^{p, q}_r \Rightarrow 0$ then $'E_r^{p,q} \Rightarrow 0$ as well. The $'E_0$-page is given by 
\[
\begin{tikzcd}
0  & D                   & E                         & F                       & 0 \\
0 & A \ar{u}{\alf} & B \ar{u}{\beta} & C \ar{u}{\gam} & 0
\end{tikzcd}
\]
and therefore the $'E_1$-page is the cohomology of the $'E_0$-page, with induced differentials of bidegree $(r, 1-r) = (1, 0)$, and is given by
\[
\begin{tikzcd}
0 \ar{r}  & \mbox{coker}(\alf) \ar{r} & \mbox{coker}(\beta)\ar{r} & \mbox{coker}(\gam)\ar{r} & 0 \\
0 \ar{r}  & \ker{(\alf)} \ar{r}             & \ker{(\beta)} \ar{r}             & \ker{(\gam)} \ar{r}            & 0
\end{tikzcd}
\]
Now, the $'E_2$-page has differentials of bidegree $(r, 1-r) = (2, -1)$, and so looks like the following diagram
\[
\begin{tikzcd}
0 \ar{drr} & 0 \ar{drr} & 0   \ar{drr} &                  &                   &      & \\
0 \ar{drr} & 0 \ar{drr} &?? \ar{drr} & ? \ar{drr} & ?   \ar{drr} & 0     & \\
                 & 0              & ?  \ar{drr} & ? \ar{drr} & ?? \ar{drr} & 0    & 0 \\
                 &                 &                    &                 & 0                 & 0  & 0
\end{tikzcd}
\]
Now, since the points marked with a single question mark $?$ all have $0$ maps as differentials then the cohomology stabilizes, and are therefore isomorphic to their respective points in the $'E_\infty$-page. But we've already shown that the $'E_\infty$-page is $0$, and so the single question marks are all $0$. This means that the sequences in the $'E_1$-page are exact everywhere except at $\mbox{coker}(\alf)$ and $\ker{(\gam)}$.

The cohomology of the double question marks $??$ stabilizes in the $'E_3$-page, and therefore 
\[
\begin{tikzcd}
0\ar{r} & ?? \ar{r} & ??\ar{r} & 0
\end{tikzcd}
\]
is an exact sequence, which implies that they are isomorphic to one another. The left hand $??$ is $\ker{(\mbox{coker}(\alf)\to\mbox{coker}(\beta))}$ and the right hand $??$ is $\mbox{coker}(\ker{(\beta)}\to\ker{(\gam)})$, that is
\[
	\ker{(\mbox{coker}(\alf)\to\mbox{coker}(\beta))} \simeq \mbox{coker}(\ker{(\beta)}\to\ker{(\gam)})
\]
Combining all this information we obtain an exact sequence 
\[
\begin{tikzcd}[cramped, sep=small]
0 \ar{r} & \ker{(\alf)} \ar{r} & \ker{(\beta)} \ar{r} & \ker{(\gam)} \ar{r}{\del} & \mbox{coker}(\alf) \ar{r} & \mbox{coker}(\beta) \ar{r} & \mbox{coker}(\gam) \ar{r} & 0
\end{tikzcd}
\]
where $\del$ is the canonical projection from $\ker{(\gam)}$ to $\mbox{coker}(\ker{(\beta)}\to\ker{(\gam)})$. 
\end{ex}

\section{The Spectral Sequence of a Polycomplex}
\label{new spec seq}
The spectral sequence of a bicomplex is a classical and typical construction found in most texts that introduce spectral sequences (see~\cite{ca},~\cite{McClearyJohn2001Augt} or~\cite{weibel} as representative examples), however we will also require the spectral sequences that arise from a tricomplex, which has not been explored in the literature nearly as deeply (if at all).  Here we introduce a new generalisation of the construction of the spectral sequence of a bicomplex, and show how to construct a spectral sequence for any higher complex.  

A graded vector space $V$ indexed by elements of $\mb{Z}^k$ will be referred to as \emph{$k$-graded}, that is, we can represent $V$ as a direct sum
\[
	V = \bigoplus_{x_i\in\mb{Z}} V^{(x_1,\dots ,x_k)}
\]
We will focus on $k$-graded vector spaces such that each grading is non-negative, that is, each $x_i\ge 0$.  A $k$-graded linear map $\map{f}{V}{W}$ is said to have \emph{polydegree $(d_1, \dots, d_k)$} where each $d_i\in\mb{Z}$ if 
\[
	f(V^{(x_1, \dots, x_k)}) \subseteq W^{(x_1 + d_1, \dots, x_k+ d_k)}
\]
for every $(x_1, \dots, x_k)\in\mb{Z}^k$.

\begin{defn}
A \emph{polycomplex of dimension $k$} (or just $k$-complex) for $k\ge1$ is a $k$-graded vector space $V$ equipped with $k$ linear maps $\map{\partial_i}{V}{V}$ each with polydegree $d_i = 1$ and $d_j = 0$ for $j\ne i$, such that
\[
	\partial^2_i = 0
\]
for each $i\in\{1,\dots, k\}$, and 
\[
	\partial_i\partial_j + \partial_j\partial_i = 0
\]
for each $i\ne j$.
\end{defn}

A $1$-dimensional polycomplex ($1$-complex) is just a cochain complex. A polycomplex of dimension $2$ ($2$-complex) is simply a bicomplex, and in this case we will denote $\partial_h := \partial_1$ and $\partial_v := \partial_2$
\[
\begin{tikzcd}
\tvdots                                                                                        & \tvdots                                                                     & \tvdots                                                           & \\   
\arrow{r}{\partial_h} V^{(0, 2)} \arrow{u}                                         & \arrow{r}{\partial_h} V^{(1, 2)} \arrow{u}       & \arrow{r} V^{(2, 2)}  \arrow{u}           &  \cdots \\
 \arrow{r}{\partial_h} V^{(0, 1)} \arrow{u}{\partial_v}     &\arrow{r}[swap]{\partial_h} V^{(1, 1)} \arrow{u}[swap]{\partial_v}   &  \arrow{r} V^{(2, 1)} \arrow{u}[swap]{\partial_v}       & \cdots  \\
 \arrow{r}[swap]{\partial_h} V^{(0, 0)}\arrow{u}{\partial_v}            & \arrow{r}[swap]{\partial_h} V^{(1, 0)} \arrow{u}{\partial_v}          &\arrow{r} V^{(2, 0)} \arrow{u}[swap]{\partial_v}     & \cdots
\end{tikzcd}
\]
A polycomplex of dimension $3$ (or $3$-complex) will be referred to as a \emph{tricomplex} and is a $3$-dimensional grid 
\[
\begin{tikzcd}[cramped, sep=small]
\ddots &\tvdots & \ddots& \tvdots &\ddots&\tvdots & \\
& V^{(0,1,2)} \ar{rr} \ar{u}\ar{ul}     &                                  \tvdots                                  & V^{(1,1,2)} \ar{rr}\ar{ul}   \ar{u} \ar[from=dd]         &                            \tvdots                                      & V^{(2,1,2)} \ar{u} \ar{r} \ar[from=dd]   \ar{ul} & \mathrlap{\cdots}\tvdots &\\
\ddots&                                      & V^{(0,0,2)} \ar{ul} \ar[crossing over]{rr}   \ar{u}  &                              & V^{(1,0,2)} \ar{ul} \ar[crossing over]{rr} \ar{u}                                      &                                   & V^{(2,0,2)}  \ar{ul}  \ar{u}   \ar{r} & \cdots    \\
&V^{(0,1,1)} \ar{rr} \ar{uu} \ar{ul} &                                                                                  & V^{(1,1,1)} \ar[from = dd] \ar{rr}            &                                    & V^{(2,1,1)} \ar{r}\ar[from=dd]           &       \cdots                                                     \\
\ddots &                                     & V^{(0,0,1)} \ar[crossing over]{rr} \ar{ul} \ar[crossing over]{uu} &                                               & V^{(1,0,1)} \ar[crossing over]{rr} \ar{ul} \ar[crossing over]{uu} &                                   & V^{(2,0,1)} \ar{ul} \ar{uu}    \ar{r} & \cdots          \\
&V^{(0,1,0)} \ar{uu} \ar{rr}\ar{ul} &                                                                                                           & V^{(1,1,0)} \ar{rr} &                                                                                  & V^{(2,1,0)} \ar{r}&                                                       \cdots     \\
 &                                             & V^{(0,0,0)} \ar{ul}{\partial_2} \ar[crossing over, near start]{uu}{\partial_3} \ar[swap]{rr}{\partial_1} &                                               & V^{(1,0,0)} \ar{ul} \ar[crossing over]{uu} \ar{rr} &       & V^{(2,0,0)} \ar{ul} \ar{uu}\ar{r} & \cdots \\
\end{tikzcd}
\]

\begin{defn}
Let $V$ be a $k$-complex. Then the \emph{total complex of $V$} is the complex $\mc{T}(V)$ given by
\[
	\mc{T}(V)^n = \bigoplus_{\sum_{i=1}^k x_i = n}V^{(x_1, \dots, x_k)}
\]
and the differential $\map{d}{\mc{T}(V)}{\mc{T}(V)}$ is given by $d = \sum_{i=1}^n\partial_i$. If $k= 1$ then we set $\mc{T}(V) = V$ and $d = \partial_1$. 
\end{defn}

Giving a precise formula for the differential is difficult in general, due to a reliance on the ordering of the components. For a $k$-complex $V$, the $n$-th component of $\mc{T}(V)$ is made up of the direct sum of all the components $V^{(x_1, \dots, x_k)}$ such that
\[
	\sum_{i=1}^k x_i = n
\]
where $x_i\ge 0$. A basic combinatorial argument then tells us that there are $\nCr{n+k-1}{k-1}$ solutions to this sum, and therefore $\nCr{n+k-1}{k-1}$ summands making up each $\mc{T}(V)^n$. Therefore the differential $\map{d^n}{\mc{T}(V)^n}{\mc{T}(V)^{n+1}}$ can be represented by an $\nCr{n+k}{k-1}\times\nCr{n+k-1}{k-1}$ matrix (when considering elements of $\mc{T}(V)^n$ as a column vector). The precise construction of this matrix depends on the ordering of the summands in $\mc{T}(V)^n$ and $\mc{T}(V)^{n+1}$. However, each of these matrices can be transformed into one another through a finite series of column and row swap operations. A row swap operation corresponds to a change in the order of summands of $\mc{T}(V)^{n+1}$, and a column swap is a change in the ordering of $\mc{T}(V)^n$. Therefore all these matrices represent the same map $d$. 

Consider a tricomplex $V$ as an example. Then the first 3 terms of the total complex $\mc{T}(V)$ are given by
\begin{align*}
\mbox{T}(V)^0 =& V^{(0,0,0)} \\ 
\mbox{T}(V)^1 =& V^{(1,0,0)}\oplus V^{(0,1,0)}\oplus V^{(0,0,1)}\\
\mbox{T}(V)^2 =& V^{(2,0,0)}\oplus V^{(0,2,0)}\oplus V^{(0,0,2)} \oplus V^{(0,1,1)} \oplus V^{(1,0,1)}\oplus V^{(1,1,0)}.
\end{align*}
Therefore, given these orderings, the first two differentials are
\begin{align*}
	d^0 &= \left(\begin{matrix} \partial_1 \\ \partial_2 \\ \partial_3 \end{matrix}\right) \\
	d^1 &= \left(\begin{matrix} \partial_1 & 0 & 0 \\ 0 & \partial_2 & 0 \\ 0 & 0 & \partial_3 \\ 0 & \partial_3 & \partial_2 \\ \partial_3 & 0 & \partial_1 \\ \partial_2 & \partial_1 & 0 \end{matrix}\right)
\end{align*}
The differential is any possible linear combination of the differentials, but which row/column is where depends on the ordering of summands. The composition of the first two differentials is given by
\[
	d^1d^0 = \left(\begin{matrix} \partial_1^2 \\ \partial_2^2 \\ \partial_3^2 \\ \partial_3\partial_2 + \partial_2\partial_3 \\ \partial_1\partial_3 + \partial_3\partial_1 \\ \partial_2\partial_1 + \partial_1\partial_2 \end{matrix} \right) 
\]
which is $0$, by the properties of the differentials of $M$. Similarly, the properties of the differentials $\partial_i$ guarantee that each $d^{n+1}d^n = 0$ for all $n\ge 0$. 

When considering a bicomplex $V$, we easily obtain two filtrations on the total complex by restricting the size of each index, that is 
\[
	F_1^t\mc{T}(V)^n = \bigoplus_{\substack{p+q = n,\\ p\ge t}} V^{(p,q)}
\]
and 
\[
	F_2^t\mc{T}(V)^n = \bigoplus_{\substack{p+q = n,\\ q\ge t}} V^{(p,q)}
\]
Then clearly we have
\[
	F_i^0\mc{T}(V)^n = \mc{T}^n(V)
\]
and when $t>n$ then $F_i^t\mc{T}(V)^n =0$. Hence, these two filtrations are canonically bounded and give rise to bounded convergent spectral sequences of $\mc{T}(V)$.  

These filtrations can be easily generalised to an arbitrary $k$ dimensional polycomplex, one for each restriction of an appropriate combination of indices, such that each one is canonically bounded. For example, for a tricomplex $V$, indexed by $(x_1,x_2,x_3)$, there are $6$ canonically bounded filtrations, associated with the following conditions:
\[
	x_1\ge s, \quad x_2\ge s, \quad x_3\ge s, \quad x_1+x_2\ge s, \quad x_1+x_3 \ge s, \quad x_2+x_3\ge s
\]
This is equivalent to choosing $1$ index out of $3$ and $2$ indices out of $3$, so in general there are
\[
	\sum_{i=1}^{k-1}\nCr{k}{i} =  \sum_{i=0}^k\nCr{k}{i} - \nCr{k}{0} - \nCr{k}{k} = 2^k - 2
\]
canonically bounded filtrations associated with a $k$-complex $V$, thus we have the following result.

\begin{thm}
\label{spec seqs of k comp}
For every $k$-complex $V$ there are at least $2^k - 2$ convergent spectral sequences calculating $H^*(\mc{T}(V))$. 
\end{thm}

\begin{lem}
\label{polycomp spec E1}
Let $V$ be a $k$-complex for $k>1$. Fix a subset $A\subset \{x_1,\dots, x_k\}$, and let $F$ be the filtration of $\mc{T}(V)$ given by
\[
	F^s\mbox{T}(V) = \bigoplus_{\substack{\sum x_j = n,\\ \sum_{x_i\in A} x_i\ge s}}V^{(x_1, \dots, x_k)}
\]
and $E$ the spectral sequence associated with $F$. Then
\[
	{E^{p,q}_1} = H^q(\mc{T}(V[\textstyle \sum_{x_i\in A} x_i=p]))
\]
where $V[\sum_{x_i\in A} x_i=p]$ is the polycomplex of degree $k-1$ obtained by fixing $\sum_{x_i\in A} x_i = p$ in $V$. 
\end{lem}

\begin{proof}
Observe that
\[
	{E^{p, q}_0} = \frac{\displaystyle\bigoplus_{\substack{\sum x_j = p+q, \\\sum_{x_i\in A} x_i\ge p}} V^{(x_1,\dots, x_k)}}{\displaystyle\bigoplus_{\substack{\sum x_j = p+q,\\ \sum_{x_i\in A} x_i\ge p+1}} V^{(x_1,\dots, x_k)}} \simeq \bigoplus_{\substack{\sum_{j\ne i}x_j = q,\\ \sum_{x_i\in A} x_i=p}}V^{(x_1,\dots, x_k)}
\]
Then the vertical differential $\map{d_0}{{E^{p, q}_0}}{{E^{p, q+1}_0}}$ is simply $\sum_{j\ne i}\partial_j$, which is the differential of $\mc{T}(V[\sum_{x_i\in A} x_i=p])$. Therefore each column ${E^{p,*}_0}$ of the ${E_0}$-page is $\mc{T}(V[\sum_{x_i\in A} x_i=p])$. Then we obtain our result by taking cohomology.
\end{proof}

It will be beneficial to clarify the lemma above by considering the example of a tricomplex. So let $V$ be a tricomplex indexed by $V^{(x_1,x_2,x_3)}$ and consider the filtration $F$ given by 
\[
	F^s\mc{T}(V)^n = \bigoplus_{\substack{x_1+x_2+x_3 = n,\\ x_1\ge s}}V^{(x_1,x_2,x_3)}
\]
Then the column $E^{0,\ast}_0$ for example is the total complex of $V[ x_1=0]$, which in this case is the bicomplex 
\[
\begin{tikzcd}
\tvdots                                                                                        & \tvdots                                                                     & \tvdots                                                           & \\   
\arrow{r} V^{(0,0, 2)} \arrow{u}                                         & \arrow{r} V^{(0,1, 2)} \arrow{u}       & \arrow{r} V^{(0,2, 2)}  \arrow{u}           &  \cdots \\
 \arrow{r} V^{(0,0, 1)} \arrow{u}     &\arrow{r} V^{(0,1, 1)} \arrow{u}   &  \arrow{r} V^{(0,2, 1)} \arrow{u}       & \cdots  \\
 \arrow{r}[swap]{\partial_2} V^{(0, 0,0)}\arrow{u}{\partial_3}            & \arrow{r} V^{(0,1, 0)} \arrow{u}          &\arrow{r} V^{(0,2, 0)} \arrow{u}     & \cdots
\end{tikzcd}
\]
or in other words it is the total complex of the bicomplex sitting at the $x_1=0$ position of $V$. Similarly $E^{1,\ast}_0$ is the total complex of the bicomplex sitting at the $x_1=1$ position of $V$ and so on. 

Note that the $\partial_1$ differential doesn't appear in $V[x_1=0]$. This is true for all $V[x_1=p]$, and therefore the spectral sequence $E$ arising from $F$ is the spectral sequence that ``forgets'' the $\partial_1$ differential. 

Now consider the filtration $F$ given by
\[
	F^s\mc{T}(V)^n =  \bigoplus_{\substack{x_1+x_2+x_3 = n,\\ x_1+x_2\ge s}}V^{(x_1,x_2,x_3)}
\]
Then $V[x_1+x_2=1]$ for example is the bicomplex  
\[
\begin{tikzcd}
\tvdots                                                          & \tvdots                          \\   
 V^{(0,1, 2)} \arrow{u}                                & V^{(1,0, 2)} \arrow{u}        \\
 V^{(0,1, 1)} \arrow{u}                               &V^{(1,0, 1)} \arrow{u}    \\
  V^{(0, 1,0)}\arrow{u}{\partial_3}            &V^{(1,0, 0)} \arrow{u}{\partial_3}          
\end{tikzcd}
\]
which is the bicomplex in $V$ that intersects with the $x_1+x_2=1$ plane, and we see that the differentials $\partial_1$ and $\partial_2$ have been forgotten. And so for this filtration $E^{p,q}_0$ is $\mc{T}(V[x_1+x_2=p])^q$ and is given by
\[
	\mc{T}(V[x_1+x_2=p])^q = \bigoplus_{\substack{x_3 = q,\\ x_1+x_2= p}}V^{(x_1,x_2,x_3)}
\]
Summarising these examples in a more general setting, we can say that when applying the filtration $F$ coming from the condition 
\[
	\sum x_i \ge s
\]
for some subset of the indices $\{x_1,\dots, x_k\}$, then for any $k$-complex $V$ the column $E_0^{p,\ast}$ is the total complex of $(k-1)$-complex that intersects with the plane $\sum x_i = p$. We also observe that the differentials $\partial_i$ are all forgotten, and therefore each of the $2^k-2$ spectral sequences of $V$ are the spectral sequences that arise from all the different possible combinations of forgetting the differentials and their respective $(k-1)$-complexes. 

\begin{prop}
\label{spec seq morphs prop}
Let $V$ be a $k$-complex where $k\ge 3$ and consider the filtration $F$ on $V$ coming from the condition
\[
	\sum_{r=1}^m x_{i_r}\ge t
\]
and the filtration $F^\prime$ coming from the condition
\[
	\sum_{r=1}^n x_{j_r}\ge t
\]
such that $n>m$ and $\{x_{i_1},\dots, x_{i_m}\}\subset\{x_{j_1},\dots,x_{j_n}\}$, then there is a morphism of spectral sequences $\map{f}{E}{E^\prime}$ where $E$ and $E^\prime$ are the spectral sequences arising from $F$ and $F^\prime$ respectively.
\end{prop}

\begin{proof}
Since $\{x_{i_1},\dots, x_{i_m}\}\subset\{x_{j_1},\dots,x_{j_n}\}$ then
\[
	\sum_{r=1}^n x_{j_r} = \sum_{r=1}^m x_{i_r} + \sum x_{j_r}
\]
where $\sum x_{j_r}$ is the sum of the indices $x_{j_r}\in \{x_{j_1},\dots,x_{j_n}\}\setminus\{x_{i_1},\dots, x_{i_m}\}$ and therefore 
\[
	\sum_{r=1}^m x_{i_r}\ge t \ge t - \sum x_{j_r}
\]
So if indices $(x_1,\dots, x_k)$ satisfy $\sum_{r=1}^m x_{i_r}\ge t$ then they also satisfy $$\sum_{r=1}^n x_{j_r}\ge t$$ And therefore $F^t\mc{T}(V)^n\subset {F^\prime}^t\mc{T}(V)^n$.

Now, consider the inclusion $\map{i}{(T(V), F)}{(T(V), F^\prime)}$ as a morphism of filtered cochain complexes, then by Proposition~\ref{spec morph prop} this induces a morphism of spectral sequences $\map{f}{E}{E^\prime}$.
\end{proof}

\bibliographystyle{abbrv}
\bibliography{spec_seq_of_polycomp_ref}
\end{document}